\documentclass[11pt]{amsart}
\usepackage{amsmath,amssymb,amsthm,upref,graphicx,mathrsfs}

\usepackage{color,xcolor}
\usepackage[
  colorlinks=true,
  linkcolor=blue,
  citecolor=blue,
  urlcolor=blue]{hyperref}

\numberwithin{equation}{section}

\newtheorem{thm}{Theorem}[section]
\newtheorem{cor}[thm]{Corollary}
\newtheorem{lem}[thm]{Lemma}
\newtheorem{prop}[thm]{Proposition}
\newtheorem{defn}[thm]{Definition}

\numberwithin{equation}{section}

\begin{document}

\leftline{ \scriptsize}

\vspace{1.3 cm}
\title
{ On the number of near-vector spaces determined by finite fields}
\author{ Kijti Rodtes and  Wilasinee Chomjun }
\thanks{{\scriptsize
\newline Keywords: Finite fields, Near-vector spaces}}
\hskip -0.4 true cm

\maketitle


\begin{abstract}  A slip on a paper concerning near-vector spaces is fixed.  New characterization of near-vector spaces determined by finite fields is provided and the number (up to the isomorphism) of these spaces is exhibited.
\end{abstract}

\vskip 0.2 true cm


\pagestyle{myheadings}
\markboth{\rightline {\scriptsize Kijti Rodtes and  Wilasinee Chomjun}}
         {\leftline{\scriptsize }}
\bigskip
\bigskip


\vskip 0.4 true cm

\section{Introduction}
In 1974,  Andre introduced and studied the concept of near-vector spaces. Later several researchers,  for example, Van der walt, Howell, Mayer and Tim Boykett, paid attention to investigate such concept.  In 2010, Howell and Mayer classified near-vector spaces over finite fields of $p$ ($p$ is prime ) elements up to isomorphism. They also extended the result to a finite field of  $p^n$ elements in Theorem 3.9, \cite{KHM}. This theorem assertes that the number of near-vector spaces $V=\mathbb{\mathbb{F}}^{\oplus m}$ over a finite field $\mathbb{F}=GF(p^n)$ is exactly $$\left ( \begin{array}{c} m+\tfrac{\phi(p^{n}-1)}{n}-2 \\ m-1 \end{array}\right )$$ up to the isomorphism in Definition \ref{defiso}, where $\phi$ is the Euler's totient function . This number is calculated based on the number of distinct suitable sequences, (Definition 2.3 in \cite{KHM}). Namely, if $A^{\ast}_{1}$ and $A^{\ast}_{2}$ are determined by suitable sequences $(S_1)$ and $(S_2)$, respectively, then $(\mathbb{F}^{\oplus m},A^{\ast}_{1}) \cong (\mathbb{F}^{\oplus m},A^{\ast}_{2})$ if and only if  $(S_1)=(S_2).$ \\

However, for the case $m=4, n=3$ and $p=3,$ it turns out that $(\mathbb{F}^4,A^{\ast}_{1}) \cong (\mathbb{F}^4,A^{\ast}_{2})$, where $A^{\ast}_{1}=\{s_\alpha \,|\, \alpha \in \mathbb{F} \}$ and $A^{\ast}_{2}=\{t_\beta \,|\, \beta \in \mathbb{F} \}$ are constructed using the sequences $(S_1) = (1,1,5,5)$ and  $(S_2) = (1,1,7,7)$, respectively, which are distinct suitable sequences. Precisely,  the isomorphism is obtained by the group isomorphisms $\theta: \mathbb{F}^{\oplus 4} \longrightarrow \mathbb{F}^{\oplus 4}$ defined by $\theta(x_1,x_2,x_3,x_4):= (x_3,x_4,x_{1}^{9},x_{2}^{9})$ and $\eta:A^{\ast}_{1}\longrightarrow A^{\ast}_{2}$ defined by $\eta(s_\alpha):= t_{\alpha ^5}$, which can be seen that;
 \begin{eqnarray*}
\theta ((x_1,x_2,x_3,x_4)s_\alpha) & =& \theta(x_1\alpha,x_2\alpha,x_3\alpha^{5},x_4\alpha^{5})\\
   & =& (x_3\alpha^5,x_4\alpha^5,x_{1}^{9}\alpha^{9},x_{2}^{9}\alpha^{9})  \\
   & = &(x_3\alpha^5,x_4\alpha^5,x_{1}^{9}(\alpha^5)^7,x_{2}^{9}(\alpha^5)^7)\\
   & = & (x_3,x_4,x_{1}^{9},x_{2}^{9})t_{\alpha^5}\\
  & = &  \theta (x_1,x_2,x_3,x_4)\eta(s_\alpha)
\end{eqnarray*}
for all $(x_1,x_2,x_3,x_4)\in \mathbb{F}^{\oplus 4},s_\alpha\in A^{\ast}_{1}$.
This contradicts to the main results of the paper. A slip can be found in the proof of Theorem 3.9 in \cite{KHM} (line 17th in the proof) in which there is using the isomorphism $\eta$ to be $\eta(s_\alpha)=t_\alpha.$ In fact, this should be $\eta(s_\alpha)=t_{\alpha^q},$ for some $1\leq q \leq p^{n}-1$ and $\gcd(q,p^{n}-1)=1.$ \\

In this article, the slip is fixed and a criteria for the classification of near-vector spaces $\mathbb{F}^{\oplus m}$ over a finite field $\mathbb{F}=GF(p^n)$ is provided.  The number of near-vector spaces up to the isomorphism is also displayed based on subgroups lattice of the abelian group $G:=U(p^n-1)/\left< p \right>$.

  \section{Preliminary}
Let $p$ be a prime, $n$ be a positive integer and  $\mathbb{F}=\operatorname{GF}(p^n)$, be a field of $p^n$ elements.  We first recall the definition of a near-vector space over a finite field $\mathbb{F}$.
\begin{defn}\label{defnear}(\cite{JA}, cf. \cite{KHM}.) A pair $(V , A)$ is called a \textbf{near-vector space} if:
\begin{enumerate}
  \item $(V , +)$ is a group and $A$ is a set of endomorphisms of $V$;
  \item $A$ contains the endomorphisms $0$, $id$ and $-id$;
  \item $A^{\ast }= A \backslash \{0\}$ is a subgroup of the group Aut$(V )$;
  \item $ A$ acts fixed point freely on $V$; i.e., for $ x \in V$  and $\alpha ,\beta \in A, x\alpha = x\beta$ implies that $x = 0$ or
$\alpha = \beta$;
  \item The quasi-kernel $Q (V )$ of $V$, generates $V$ as a group. Here, $Q (V ) = \{x \in V : \forall \alpha, \beta \in  A, \exists \gamma \in A \, $  such that  $x\alpha + x\beta = x\gamma \}$.
\end{enumerate}
\end{defn}
The $dimension$ of the near-vector space, $\dim(V )$, is uniquely determined by the cardinality of an independent generating set for $Q (V )$.

\begin{defn}\label{defiso}(\cite{KHM}) Two near-vector spaces $(V_1, A_1)$ and $(V_2, A_2)$ are \textbf{isomorphic}, written by $(V_1, A_1)\cong (V_2, A_2)$, if there are group isomorphisms  $\theta : (V_1,+)\rightarrow (V_2,+)$ and $\eta	:( A^{\ast} _{1},\cdot)\rightarrow(A^{\ast} _{2},\cdot )$ such that $\theta(x\alpha)=\theta(x)\eta(\alpha)$  for all $x\in V_1$ and $\alpha\in A^{\ast}_{1}$.
\end{defn}
In fact, the group isomorphism $\eta	:( A^{\ast} _{1},\cdot)\rightarrow(A^{\ast} _{2},\cdot )$ can be extended to a semigroup isomorphism $\widehat{\eta}:A_1\longrightarrow A_2$ by setting $\widehat{\eta}(0)=0$ and $\widehat{\eta}(\alpha)=\eta(\alpha)$, for all $\alpha\in A^{\ast} _{1} $.  For a near-vector space $(V , A)$, the endomorphisms in $A$ are sometimes called \emph{scalars} and the action of these endomorphisms on the elements of $V$ is sometimes called \emph{scalar multiplication}.  \\

Recall that a (right) \emph{near-field} is a triple $(F,+,\cdot)$ that satisfies all the axiom of a skew-field, except perhaps the left distributive law, \cite{KHM}.  By, Theorem 3.5 in \cite{KHM}, any finite dimensional near-vector space $(V,A)$ can be decomposed as $\bigoplus_{i=1}^{m} (F_i, F_i)$, where $m=\dim(V)$ and $F_i$'s are near-fields.  Namely, there exist semigroup isomorphisms $\psi_i:(A,\circ)\longrightarrow (F_i,\cdot)$ and an additive group isomorphism $\phi:V\longrightarrow F_1\oplus \cdots \oplus F_m$ such that if $\phi(v)=(x_1,\dots, x_m)$ then $\phi(v\alpha)=(x_1\psi_1(\alpha),\dots,x_m\psi_m(\alpha))$, for all $v\in V, \alpha\in A$.\\

In the case of $F_i=\mathbb{F}$ for all $i=1,\dots m$, all such near-vector spaces, which we now call \emph{near-vector spaces over a finite field}, are determined by semigroup automorphisms $\psi_i:(\mathbb{F},\cdot)\longrightarrow (\mathbb{F},\cdot)$.  Precisely, for a near-vector space $(V,A)$ with $V=\mathbb{F}^{\oplus m}:=\mathbb{F}\oplus \cdots \oplus\mathbb{F}$ and $A=\{s_\alpha \,|\, \alpha\in \mathbb{F} \}$, the scalar multiplication on $V$ is given by
$$ (x_1,\dots, x_m)s_\alpha=(x_1\psi_1(\alpha),\dots,x_m\psi_m(\alpha)), $$
for every $\alpha \in \mathbb{F}$.  \\

Thus, the classification of near-vector spaces over finite fields depends on semigroup automorphisms $\psi_i$'s and their actions.  By Proposition 2.5 in \cite{KHM}, each $\psi_i$ is given by $\psi_i(x)=x^{q_i},  \forall x\in \mathbb{F}^\ast$, for some $q_i\in U(p^n-1):=\{ 1\leq q \leq p^n-1 \,|\, \operatorname{gcd}(q,p^n-1)=1\}$ (multiplicative group).  Moreover, for $A=\{ s_\alpha\,|\, \alpha \in \mathbb{F}\}$, $\widetilde{A}_\sigma=\{ \sigma_\alpha\,|\, \alpha \in \mathbb{F}\}$ and $\overline{A}=\{ t_\alpha\,|\, \alpha \in \mathbb{F}\}$, if the actions of $A, \widetilde{A}$ and $\overline{A}$ on $\mathbb{F}^{\oplus m}$ are given respectively by
\begin{eqnarray*}
  (x_1,\dots,x_m)s_\alpha &:=& (x_1\alpha^{q_1},\dots,x_m\alpha^{q_m}), \\
  (x_1,\dots,x_m)\sigma_\alpha &:=& (x_1\alpha^{q_\sigma(1)},\dots,x_m\alpha^{q_\sigma(m)}), \hbox{ and } \\
  (x_1,\dots,x_m)t_\alpha &:=& (x_1\alpha^{q_1p^{l_1}},\dots,x_m\alpha^{q_mp^{l_m}}), 
\end{eqnarray*}
where $q_i\in U(p^n-1)$, $l_i\in \{0,1,\dots,n-1 \}$ and $\sigma$ is a permutation in $S_m$, then $$(\mathbb{F}^{\oplus m},A)\cong (\mathbb{F}^{\oplus m},\widetilde{A}) \quad \hbox{ and } \quad  (\mathbb{F}^{\oplus m},A)\cong (\mathbb{F}^{\oplus m},\overline{A}).$$  In fact, the first isomorphism is derived by  $$\theta_1: \mathbb{F}^{\oplus m}\longrightarrow \mathbb{F}^{\oplus m}; (x_1,\dots,x_m)\mapsto (x_{\sigma(1)},\dots, x_{\sigma(m)}) \hbox{  and }\eta_1: A^\ast \longrightarrow \widetilde{A}^\ast; s_\alpha\mapsto \sigma_\alpha,$$ (see Lemma 3.6 in \cite{KHM}).  The second isomorphism is derived by $$\theta_2: \mathbb{F}^{\oplus m}\longrightarrow \mathbb{F}^{\oplus m}; (x_1,\dots,x_m)\mapsto (x^{p^{l_1}}_{1},\dots, x^{p^{l_m}}_{m}) \hbox{ and } \eta_2: A^\ast \longrightarrow \widetilde{A}^\ast; s_\alpha\mapsto t_\alpha,$$ (confront Lemma 3.7 in \cite{KHM}).  This motivates us to consider the group $G:=U(p^n-1)/\left< p  \right>$, where the operation  is the usual multiplication modulo $p^n-1$ and $\left< p \right>=\{1,p,\dots,p^{n-1} \}$.  By the above discussion, we can identify the action of  $A$ on $\mathbb{F}^{\oplus m}$ by a non-decreasing sequence  of length $m$ on $G$.  Precisely, if $A=\{s_\alpha \,|\, \alpha\in \mathbb{F} \}$ acts on $\mathbb{F}^{\oplus m}$ by $$(x_1,\dots,x_m)s_\alpha = (x_1\alpha^{q_1},\dots,x_m\alpha^{q_m}),$$ then we identify this action by the sequence $$(S):=(q_1,\dots,q_m).$$   A non-decreasing sequence $(S)=(q_1,\dots,q_m)$ in which $q_i$ is the smallest in the class $\overline{q}_i\in G$ for each $i=1,\dots,m$, is called a \emph{suitable sequence of length $m$} (confront Definition 2.3 in \cite{KHM}). \\
 
 Furthermore, if $q\in G$ and $(S)=(q_1,\dots,q_m)$ is a suitable sequence of length $m$ on $G$ and $A=\{s_\alpha \,|\, \alpha\in \mathbb{F} \}$ is identified by $(S)$, then $ (\mathbb{F}^{\oplus m},A)\cong (\mathbb{F}^{\oplus m},A')$, where $A'=\{s'_\alpha \,|\, \alpha\in \mathbb{F} \}$ is identified by $(S'):=q(S):=(qq_1,\dots,qq_m)$.  Here, the isomorphism is derived by $$\theta': \mathbb{F}^{\oplus m}\longrightarrow \mathbb{F}^{\oplus m}; (x_1,\dots, x_m)\mapsto (x_1,\dots, x_m) \hbox{  and }\eta': A^\ast \longrightarrow A'^\ast; s_\alpha\mapsto s'_{\alpha^{1/q}},$$
(see discussion before Theorem 3.9 in \cite{KHM}).  Therefore, to classify and count (up to the isomorphism) all near-vector spaces $\mathbb{F}^{\oplus m}$ over a finite filed $\mathbb{F}$, it is enough to deal with the set of all suitable sequences on $G$ of length $m$ with $1$ in the first position, and we denote this set by $St(1,m,G)$.  We also define an equivalent relation $\sim$ on the set $St(1,m,G)$ by,  for suitable sequences $(S_1)$ and $(S_2)$ in $St(1,m,G)$,
\begin{equation}\label{relationimp}
    (S_1)\sim (S_2) \Longleftrightarrow (\mathbb{F}^{\oplus m},A_1)\cong (\mathbb{F}^{\oplus m},A_2),
\end{equation}
where $A_1$ and $A_2$ are determined by $(S_1)$ and $(S_2)$ respectively.

\section{Classification of Near vector spaces}
Let $G:=U(p^n-1)/\left< p \right>$, where $\left< p \right>=\{1,p,\dots,p^{n-1}\}$.  We can represent the group $G$ explicitly as $G=\{1,q_2,\dots,q_{\frac{\phi(p^n-1)}{n}} \}\subseteq U(p^n-1)$, where $q_i$ is the smallest element in the coset of $\left< p \right>$ containing $q_i$.  So, the product of $q_i$ and $q_j$ in $G$ will be $q_k\in G$ whose its coset $\overline{q_k}:=q_k\left< p \right>$ contains the remainder of $q_iq_j$ divided by $p^n-1$.  We now rewrite Theorem 2.2 in \cite{KHM} as:
\begin{thm}\label{thmm22} [\cite{KHM}] Let $\mathbb{F}=\operatorname{GF}(p^n)$ and $q_i, q_j\in G$.  Then, $(\alpha^{q_i}+\beta^{q_i})^{q_j}=(\alpha^{q_j}+\beta^{q_j})^{q_i}$ for all $\alpha, \beta \in \mathbb{F}$ if and only if $q_i=q_j$.
\end{thm}
This is a main tool used to prove the following result.  For a given suitable sequence  $(S)=(1,q_2,\dots,q_m)$, we denote $S:=\{ 1,\dots, q_N\}$ the order set (strictly increasing) of all distinct elements in $(S)$.
\begin{thm}\label{mainnewthm}
Let $(\mathbb{F}^{\oplus m}, A_1)$ and $(\mathbb{F}^{\oplus m}, A_2)$ be  near-vector spaces, where $A_1$ and $A_2$ are determined by suitable sequence
$(S_1)$ and $(S_2)$, respectively. Then $(\mathbb{F}^{\oplus m},A_1)\cong (\mathbb{F}^{\oplus m},A_2)$ if and only if
there is $q\in S_1$  such that  $S_1=qS_2$ and the occurrences of $qq'_j\in (S_1)$ and  $q'_j\in (S_2)$ are the same for each $j=1,\dots, N$, where $N=|S_1|=|S_2|$.
\end{thm}

\begin{proof}   Suppose that $(\mathbb{F}^{\oplus m},A_1)\cong (\mathbb{F}^{\oplus m},A_2)$.  So, there must exist group isomorphisms
  $\theta:\mathbb{F}^{\oplus m}	\longrightarrow \mathbb{F}^{\oplus m}$  and $\eta:A^{\ast}_{1}\longrightarrow	A^{\ast}_{2}$ such that
  \begin{equation*}
 \theta((x_1, x_2 , . . . , x_m)s_\alpha)	 =\theta(x_1,x_2, . . . ,x_m)\eta(s_\alpha).
\end{equation*} 
Suppose $(S_1)=(1=q_1,q_2, \dots ,q_m)$ and $(S_2)=(1=q'_1,q^{\prime}_{2}, \dots ,q^{\prime}_{m})$.  Then the actions of $A_1, A_2$ on $\mathbb{F}^{\oplus m}$ are explicitly given by
  $$ (x_1,x_2, . . . ,x_m) s_\alpha = (x_1\alpha^{q_1},x_2\alpha^{q_2}, . . . ,x_m\alpha^{q_m}) ,$$  
   $$ (x_1,x_2, . . . ,x_m) t_\gamma = (x_1\gamma^{q^{\prime}_{1}},x_2\gamma^{q^{\prime}_{2}}, . . . ,x_m\gamma^{q^{\prime}_{m}}).$$
for each $\alpha, \gamma \in \mathbb{F}$ and each $(x_1,x_2,\dots,x_m)\in \mathbb{F}^{\oplus m}$. 

Since $\mathbb{F}^*$ is a cyclic group, $\mathbb{F}^*=\left< a \right>$ for some $a\in \mathbb{F}^*$.   Suppose  $\eta(s_a)=t_b$, for some $b\in \mathbb{F}^*=\left< a \right>$. Thus, there exists an integer $1\leq q_\circ < p^n-1$ such that $b=a^{q_\circ}$. Since $\eta$ is an isomorphism, we have $\operatorname{ord}(s_a)=\operatorname{ord}(t_b):=k$.  So, $$(1,1,\dots,1)=(1,1,\dots,1)(t_b)^k=(1,1,\dots,1)t_{b^k}=(b^{kq'_1},b^{kq'_2},\dots,b^{kq'_m}),$$
 and then $b^k=1$ (because $\gcd (q'_i,p^{n}-1)=1$).  If $\operatorname{ord}({b})=l<k$, then $$(t_b)^l=t_{b^l}=t_1=id.$$  This implies that $\operatorname{ord}(t_b)\leq l <k$, which is a contradiction.  So, $\operatorname{ord}({b})=k$.  Similarly, $\operatorname{ord}({a})=k$ and thus $k=|\mathbb{F}^*|=p^n-1$.  Since $b=a^{q_\circ}$ and $a,b$ have the same orders,  $\gcd(q_\circ,p^n-1)=1$; namely, $q_\circ\in U(p^n-1)$.   Therefore, for the given isomorphism $\eta$, there must exist $q_\circ\in U(p^n-1)$ such that, for each non zero $\alpha(=a^{r})$ in $\mathbb{F}$,
 $$ \eta(s_\alpha)=\eta(s_{a^{r}})=(\eta(s_a))^{r}=(t_{a^{q_\circ}})^r=t_{(a^{q_\circ})^r}=t_{\alpha^{q_\circ}}. $$

  Let $S_1=\{1,\dots, q_{N_1} \}\subseteq G$ be the order set (strictly increasing) of all elements in the sequence $(S_1)$.   Suppose that, for each $i=1,\dots,N_1$, the occurrence of $q_i$ in $(S_1)$ is $l_{1i}$.  Then $l_{1i}\geq 1$ for each $i=1,\dots,N_1$ and
  \begin{equation}\label{eqq1}
  l_{11}+\cdots+l_{1N_1}=m.
  \end{equation}

  Now, for each $1\leq k \leq N_1$, consider the constant subsequence  $(q_k,\dots, q_{k})$ (length $l_{1k}$) of $(S_1)$.   Let $e_i= (0,\dots , 1, \dots, 0)$, with $1$ in position $i$, and zeros elsewhere, $1\leqslant i\leqslant m.$ Suppose		
$$\theta(e_i ) = (\omega_{i,1}, \dots ,\omega_{i,m})$$																	
for some $\omega_{i, j}\in \mathbb{F} , k\leqslant i\leqslant k + l_{1k}, 1\leqslant j\leqslant m.$ Also, for $\alpha \in \mathbb{F}$, we have
\begin{equation*}
\theta(e_{i}s_{\alpha})=\theta(0, \dots ,\alpha^{q_i}, \dots,	0),
\end{equation*}
\begin{equation*}
\theta(e_{i})\eta(s_{\alpha})= (\omega_{i,1}, \dots , \omega_{i,m} )t_{\alpha^{q_\circ}}=(\omega_{i,1}\alpha^{q_\circ q^{\prime}_1} , \dots ,\omega_{i,m}\alpha^{q_\circ q^{\prime}_m})
\end{equation*}
and thus 
$$\theta(0, \dots ,\alpha^{q_i}, \dots,	0)=\theta(e_{i}s_{\alpha})=(\omega_{i,1}{\alpha^{q_\circ q^{\prime}_1}} , \dots,\omega_{i,m}{\alpha^{q_\circ q^{\prime}_m}})$$
for $k\leqslant i\leqslant k + l$, and $\alpha^{q_i}$ in the $i^{\operatorname{th}}$ position. Hence, with $ \alpha$ in the $k^{\operatorname{th}}$ position, 
$$\theta(0, \dots, \alpha, \dots , 0)=\theta(e_{k}s_{{\alpha}^{1/q_k}}) =  (\omega_{k,1}{{\alpha}^{q_\circ q^{\prime}_{1} /q_k}} , \dots , \omega_{k,m}{{\alpha}^{q_\circ q^{\prime}_{m} /q_k}}) .$$	
Consequently, for $\alpha,\beta \in \mathbb{F},$	
$$\theta(0, \dots, \alpha+\beta, \dots , 0)=\theta(e_{k}s_{(\alpha+\beta)^{1/q_k}} )=(\omega_{k,1}{{(\alpha+\beta)}^{q_\circ q^{\prime}_{1}/q_k}} , \dots,\omega_{k,m}{{(\alpha+\beta)}^{q_\circ q^{\prime}_{m} /q_k}}).$$
We also have
\begin{align*}
\theta(0, \dots, \alpha+\beta, \dots , 0) &= \theta(e_{k} s_{{\alpha}^{1/q_k}} ) + \theta(e_{k}s_{{\beta}^{1/q_k}} )\\
&=(\omega_{k,1}{{\alpha}^{q_\circ q^{\prime}_{1}/q_k}} , \dots,\omega_{k,m}{{\alpha}^{q_\circ q^{\prime}_{m} /q_k}}) + (\omega_{k,1}{{\beta}^{q_\circ q^{\prime}_{1}/q_k}} , \dots, \omega_{k,m}{{\beta}^{q_\circ q^{\prime}_{m} /q_k}})\\
&=(\omega_{k,1}{({\alpha}^{q_\circ q^{\prime}_{1}/q_k}+{\beta}^{q_\circ q^{\prime}_{1}/q_k})} , \dots , \omega_{k,m}{({\alpha}^{q_\circ q^{\prime}_{m} /q_k}+{\beta}^{q_\circ q^{\prime}_{m} /q_k})}).
\end{align*}		
Since $e_k$ is non-zero, at least one of $\omega_{k,1}, \dots , \omega_{k,m}$ is non-zero, say $\omega_{k,r}\neq 0,$ where $r$ is minimal with respect to this property. Then, we have
$$\omega_{k,r}(\alpha^{q_\circ q^{\prime}_r /q_k}+\beta^{q_\circ q^{\prime}_r /q_k}) = \omega_{k,r}(\alpha +\beta)^{q_\circ q^{\prime}_{r}/q_k}.$$
Thus, 
\begin{equation}\label{eq 32}
    \alpha^{q_\circ q^{\prime}_r /q_k}+\beta^{q_\circ q^{\prime}_r /q_k}= (\alpha +\beta)^{q_\circ q^{\prime}_{r}/q_k}
\end{equation}
for all $\alpha,\beta\in \mathbb{F}.$ 

By Theorem \ref{thmm22}, the equation (\ref{eq 32}) happens if and only if $\frac{q_\circ q'_r}{q_k}\in \overline{1}=\left< p \right>$; equivalently, if and only if $qq'_r=q_k$, where $q\in G$ such that $q_\circ\in q\left< p \right>$.  This also implies that $\omega_{k, j} = 0$ if $qq^{\prime}_j\neq  q_{k}$, for each $j=1,\dots,m$.  Assume that $(q^{\prime}_{r},\dots ,q^{\prime}_{r+l^{\prime}})$ is the constant subsequence of maximal length of the sequence $(S_2)$ and satisfies
$q_k = qq^{\prime}_{r}= \cdots =qq^{\prime}_{r+l^{\prime}}$. Then
$$\theta(e_i)=(0, \dots , 0, \omega_{i,r} , \dots , \omega_{i,r+l^{\prime}} , 0, \dots , 0),$$	
for each $i=k,\dots,k+l_{1k}$.  If $l'<l_{1k}$, then $\{ \theta(e_k),\dots,\theta(e_{k+l_{1k}})\}$ is a linearly dependent set in the vector space $\mathbb{F}^{\oplus m}$ over $\mathbb{F}$.  So, there exists $\alpha^{q_\circ}_0,\dots,\alpha^{q_\circ}_{l_{1k}}\in \mathbb{F}$, not all zero, such that
$$0=\sum_{i=0}^{l_{1k}} \theta(e_{k+i})t_{\alpha^{q_\circ}_i} =\sum_{i=0}^{l_{1k}} \theta(e_{k+i}s_{\alpha_i})=\theta \left(\sum_{i=0}^{l_{1k}}e_{k+i}s_{\alpha_i}\right).$$
Then, $0=\sum_{i=0}^{l_{1k}}e_{k+i}s_{\alpha_i}=\sum_{i=0}^{l_{1k}}v_i$, where $v_i= (0,\dots,0,\alpha_i^{q_{k}},0,\dots,0)$, with $\alpha_i^{q_{k}}$ in position $k+i$ and zero elsewhere, $0\leq i\leq l_{1k}$.  This is a contradiction because $\{v_0,\dots,v_{l_{1k}}\}$ is a linearly independent set in the vector space $\mathbb{F}^{\oplus m}$ over $\mathbb{F}$.  Hence $l'\geq l_{1k}$.

Now, let $S_2=\{q'_1,\dots, q'_{N_2} \}\subseteq G$ be the order set (strictly increasing) of all elements in the sequence $(S_2)$.   Suppose also that, for each $j=1,\dots,N_2$, the occurrence of $q'_j$ in $(S_2)$ is $l'_{2j}$.  Then $l'_{2j}\geq 1$ for all $j=1,\dots,N_2$ and
 \begin{equation}\label{eqq2}
l'_{21}+\cdots+l'_{2N_2}=m.
  \end{equation}
Therefore, for each $q_k\in S_1$, there exist $q'_r\in S_2$ such that $qq'_r=q_k$; namely, $S_1\subseteq qS_2$ and $l_{1k}\leq l'_{2r}$.  However, by (\ref{eqq1}) and (\ref{eqq2}), we conclude that $S_1=qS_2$ and the occurrence of $q_k\in (S_1)$ is the same as occurrence of $q^{-1}q_k=q'_r\in (S_2)$.  Moreover, since $1\in S_1\cap S_2$,  $q\in S_1$ and $q^{-1}\in S_2$. \\

Conversely, we assume the assumptions and then define $\eta:A^{\ast}_{1}\longrightarrow A^{\ast}_{2}$ by $\eta(s_\alpha)=t_{\alpha^{q}}$, for each $\alpha\in \mathbb{F}$.
For  $S_1=\{ 1,\dots, q_N\}$ and $S_2=\{1, \dots,q'_N \}$, let $l_j$ and $ l'_j$ be the occurrences of $q_j\in (S_1)$ and $q'_j\in (S_2)$, respectively.  Define a permutation $\rho:\{1, . . . ,N\}\longrightarrow \{1, . . . ,N\}$ by setting $\rho(j):=n_j$ if $qq^{\prime}_{j}= q_{n_j} $, for each $j=1,\dots,q_N$.  Now, for each $k=1,\dots, N$, we set $$m_{\rho(k)}:=\sum_{j< \rho(k)}l_j \quad \hbox{ and } \quad m'_k:=\sum_{j=1}^k l'_j,$$ and set $m_1=0=m'_0$.  For each $i\in \{ 1,\dots,m \}$ such $m'_{k-1}< i \leq m'_k$, we define $$\sigma(i):=m_{\rho(k)}+j,$$
where $i=m'_{k-1}+j$.  Then $\sigma:\{1, \dots,m\}\longrightarrow \{1, \dots ,m\}$ is a permutation satisfying the property that $qq^{\prime}_{i}= q_{\sigma(i)} $, for all $i\in \{ 1,\dots, m\}$.  By setting $s_i:= qq'_i/q_{\sigma(i)} \operatorname{mod}(p^n-1)$ for each $i=1,\dots, m$, we see that $s_i\in \overline{1}=\left< p \right>$.  So, the function $\theta:\mathbb{F}^{\oplus m} \longrightarrow \mathbb{F}^{\oplus m}$ defined by $$\theta(x_1,x_2, \dots ,x_m)=(x_{\sigma(1)}^{s_1},x_{\sigma(2)}^{s_2}, \dots ,x_{\sigma(m)}^{s_m})$$ is a group isomorphism, by Theorem \ref{thmm22}.  Moreover, for $\alpha\in \mathbb{F}$ and $(x_1,x_2\dots,x_m)\in \mathbb{F}^{\oplus m}$,
\begin{align*}
\theta((x_1,x_2, \dots ,x_m)s_\alpha) &= \theta(x_1\alpha^{q_1},x_2\alpha^{q_2}, \dots ,x_m\alpha^{q_m}) \\
&=((x_{\sigma(1)}\alpha^{q_{\sigma(1)}})^{s_1},(x_{\sigma(2)}\alpha^{q_{\sigma(2)}})^{s_2}, \dots ,(x_{\sigma(m)}\alpha^{q_{\sigma(m)}})^{s_m})\\
&=(x_{\sigma(1)}^{s_1}\alpha^{q_{\sigma(1)}s_1},x_{\sigma(2)}^{s_2}\alpha^{q_{\sigma(2)}s_2}, \dots ,x_{\sigma(m)}^{s_m}\alpha^{q_{\sigma(m)}s_m})\\
&=(x_{\sigma(1)}^{s_1}\alpha^{qq^{\prime}_{1}},x_{\sigma(2)}^{s_2}\alpha^{qq^{\prime}_{2}}, \dots,x_{\sigma(m)}^{s_m}\alpha^{qq^{\prime}_{m}})\\
&=(x_{\sigma(1)}^{s_1},x_{\sigma(2)}^{s_2}, \dots ,x_{\sigma(m)}^{s_m})t_{\alpha^q}\\
&=\theta(x_1,x_2, \dots ,x_m)\eta(s_\alpha)
\end{align*}		
Therefore $(\mathbb{F}^{\oplus m},A_1)\cong (\mathbb{F}^{\oplus m},A_2)$.
\end{proof}

\section{Numbers of Near-vector spaces}
In this section, we also denote $G$ by $U(p^n-1)/ \left<p\right>$ and $1$ by the identity of $G$ as in the previous section.  The number (up to the isomorphism) of near-vector spaces $\mathbb{F}^{\oplus m}$ over a finite filed $\mathbb{F}$ is based on the number (up to the relation $\sim$ in (\ref{relationimp}) ) of equivalent classes in $St(1,m,G)$.  Let $S$ be the set of all distinct elements in a suitable sequence $(S)$ on $G$.  It turns out, by Theorem \ref{mainnewthm}, that if $|S_1|\neq |S_2|$, then $(S_1) \nsim (S_2)$.  Thus, the total number of near-vector spaces up to the isomorphism is $$\sum_{i=1}^{\min(m,|G|)} T(i),$$ where $T(i)$ denote the number  (up to $\sim$) of suitable sequences $(S)$ of length $m$ with $|S|=i$. \\

For $i=1$, it is clear that $T(1)=1$.  We now consider $T(N)$ for $2 \leq N \leq \min(m, |G|)$.   Let $St(1,m,N)$ denote the set of all possible distinct ($\neq$) suitable sequences $(S)$ having length $m$, $|S|=N$ and $1$ in the first position.  By basic combinatorics, there are $$\left(
                                                                                                                                            \begin{array}{c}
                                                                                                                                              |G|-1 \\
                                                                                                                                              N-1 \\
                                                                                                                                            \end{array}
                                                                                                                                          \right) \left(
                                                                                                                                            \begin{array}{c}
                                                                                                                                              m-1 \\
                                                                                                                                              N-1 \\
                                                                                                                                            \end{array}
                                                                                                                                          \right):=t_N
$$ suitable sequences in $St(1,m,N)$.   For a suitable sequence $(S)\in St(1,m,N)$ having the set $S=\{1=q_1, q_2,\dots,q_N \}$, by Theorem \ref{mainnewthm}, its equivalent class can be explicit as $$[(S)]:=\{ (S),q^{-1}_2(S),\dots,q^{-1}_N(S)\}.$$  Here, $q^{-1}_j(S)$ means the suitable sequence obtained from $(S)$ by multiplying elements in each position of the sequence by $q^{-1}_j$ and then rearrange their entry in non-decreasing order.  If each class in $ St(1,m,N)/\sim$ contains $N$ elements, then $T(N)=(t_N)/N$.  However, it can happen that $|[(S)]|<N$; that is $q^{-1}_i(S)=q^{-1}_j(S)$, for some $1\leq i\neq j \leq N$ or equivalently $q(S)=(S)$ for some $q\in S$ ($q=q_iq^{-1}_j\in S$, because $1\in S$).  Precisely, we have:
\begin{prop} \label{prop41}For a suitable sequence $(S)\in St(1,m,N)$,  $|[(S)]|<N$ if and only if
 \begin{enumerate}
   \item  $S$ is a disjoint union of cosets in $G/H$ such $H\subseteq S$, for some non-trivial subgroup $H$ of $G$ in which $|H|$ is a common factor of $m,N, |G|$, and
   \item Elements in $(S)$ coming from the same cosets have the same occurrences.
 \end{enumerate}
\end{prop}
\begin{proof}
 Suppose $q(S)=(S)$, for some $q\in S$.  Then, $q^kS=S$, for all $k\in \mathbb{N}$.  So,  $\left<q\right>\subseteq S$  and $q^N=1$.  Now, we can write $S=\{1, q, q^2,\dots,q^{r-1},q_{r+1}\dots,q_N \}=\left< q\right>\sqcup \{q_{r+1},\dots,q_N \}$ (disjoint union), where $r=\operatorname{ord}(q)$.  By using the fact that  $q^kS=S$, for all $k\in \mathbb{N}$ again, we can write $S=\left< q\right>\sqcup q_{r+1}\left< q\right> \sqcup \{q_{2r+1},\dots,q_N \}$.  After continuously repeating this process, we conclude that $S$ is a union of cosets in $ G/\left<q\right>$.   Moreover, since the occurrences for $qq'$ and $q'$ in $(S)$ must be equal for each $q'\in S$, elements in $S$ coming from the same coset in $ G/\left<q\right>$ must have the same occurrences.   This also implies that $m$ is divisible by $|\left<q\right>|$.

 On the other hand, suppose $H$ is a non trivial subgroup of $G$ in which its order divides $\operatorname{gcd}(m,N, |G|)$.  Let $S$ be a disjoint union of $N/|H|$-elements in $G/H$ and elements in $(S)$ coming from the same cosets have the same occurrences.  Thus, $q^{-1}(S)=(S)$, for all $q\in H$ and thus $[(S)]$ contains at most $N/|H|$ sequences.
\end{proof}
An immediate consequence of this proposition is as follow.
\begin{cor} If $\operatorname{gcd}(m,N, |G|)=1$, then $T(N)=(t_N)/N$.  In particular, if $\operatorname{gcd}(m, |G|)=1$, then the total number of near-vector space up to the isomorphism is $$\sum_{N=1}^{\min(m,|G|)} (t_N)/N$$ and
$$\left(
                                                                                                                                            \begin{array}{c}
                                                                                                                                              |G|-1 \\
                                                                                                                                              N-1 \\
                                                                                                                                            \end{array}
                                                                                                                                          \right) \left(
                                                                                                                                            \begin{array}{c}
                                                                                                                                              m-1 \\
                                                                                                                                              N-1 \\
                                                                                                                                            \end{array}
                                                                                                                                          \right)=t_N
$$ is divisible by $N$ for each $N=1,2,\dots,\min(m,|G|)$.
\end{cor}
For $\emptyset\neq H\leq G$ such that $|H|$ is a divisor of $\operatorname{gcd}(m,N,|G|)$, we denote $St(H,m,N)$ the set of all suitable sequences $(S)$ of length $m$ satisfying (1) and (2) in Proposition \ref{prop41}.
  The proof of this proposition also asserts that if $|[(S)]|<N$, then this class contains at most $N/|H|$ sequences.  In fact, we have:
\begin{prop}\label{prop43} Each $[(S)]$, where $(S)\in St(H,m,N)$, contains exactly $N/|H|$ sequences if and only if $H\nleq K$ for any subgroup $K\neq H$ of $G$ such that  $|K|$ is a divisor of $\operatorname{gcd}(m,N,|G|)$.
\end{prop}
\begin{proof} Suppose that  $H \lneq K$ for some subgroup $K$ of $G$ such that  $|K|$ is a divisor of $\operatorname{gcd}(m,N,|G|)$.  Then, each coset in $G/K$ can be written as a $|K|/|H|$-disjoint union of cosets in $G/H$.  So, there is $S\in G/H$ that also belongs to $G/K$.  This implies that $[(S)]$ contains at most $N/|K|$ sequences, which is exactly not $N/|H|$ sequences.

Conversely, suppose $|[(S)]|<N/|H|$.  Then there are disjoint cosets $[q],[q']$ in $G/H$ such that $q^{-1}(S)=q'^{-1}(S)$, or equivalently, $q(S)=(S)$ for some $q\notin H$.  Again, $q^i(S)=(S)$, for all $i\in \mathbb{N}$ and hence $\left< q\right>H\subseteq S$.  Since $G$ is abelian, $K:=\left<q\right>H$ is a subgroup of $G$.  It is clear that $H\lneq K$. Using the same process as in the proof of Proposition \ref{prop41}, we conclude that $S$ is an $N/|K|$-disjoint union of cosets in $G/K$ and $|K|$ is a common factor of $m,N, |G|$.  The proof is now completed.
\end{proof}
Proposition \ref{prop41} and the proof above, yields us to conclude that:
 \begin{cor}For a suitable sequence $(S)\in St(1,m,N)$, if $|[(S)]|\neq N$, then $|[(S)]|= N/|H|$ for some subgroup $H\leq G$ such that $|H|$ is a divisor of $\operatorname{gcd}(m,N,|G|)$.
 \end{cor}
  This motivates us to consider all possible subgroups of $G$ that their orders are divisors of $\operatorname{gcd}(m,N,|G|)$.  Let $\{d_1,\dots,d_l \}$ be the set of all divisors of $\operatorname{gcd}(m,N,|G|)$ and suppose without loss of generality that $d_1<\cdots<d_l$.  Suppose, for each $d_i$, there are $k_i$ distinct subgroups of order $d_i$; we refer to these subgroups by $H_{ij}$, for each $1\leq i \leq l$ and $1\leq j \leq k_i$.  Denote $S_{d_i}(G,N):=\{ H_{ij} \,|\,1\leq j \leq k_i \} $, for each $i=1,\dots, l$.
\begin{lem}\label{lem45} Let $H_1, H_2\in S_{d_i}(G,N)$, for some $1\leq i \leq l$ such $H_1\neq H_2$.  If there is a suitable sequence $(S)\in St(H_1,m,N)\cap St(H_2,m,N)$, then $(S)\in St(K,m,N)$ for some $K\in S_{d_j}(G,N)$ such that $d_j>d_i$ and $H_1H_2\subseteq K$.
\end{lem}
\begin{proof} Assume that  $(S)\in S(H_1,m,N)\cap S(H_2,m,N)$; i.e., $h(S)=(S)$, for all $h\in H_1 H_2$.  Since $G$ is abelian, $H_1H_2\leq G$.  By using the same arguments as in the proof of Proposition \ref{prop41}, we conclude that $(S)\in St(H_1H_2,m,N)$, which complete the proof.
\end{proof}
We observe that if $(S)\in St(H,m,N)$, then $[(S)]\subseteq St(H,m,N)$.  Thus, by Proposition \ref{prop43} and Lemma \ref{lem45}, $St(H_{lj},m,N)/\sim$ contains exactly $$\left(
                                                                                                                                            \begin{array}{c}
                                                                                                                                              \frac{|G|}{d_l}-1 \\
                                                                                                                                              \frac{N}{d_l}-1 \\
                                                                                                                                            \end{array}
                                                                                                                                          \right) \left(
                                                                                                                                            \begin{array}{c}
                                                                                                                                              \frac{m}{d_l}-1 \\
                                                                                                                                              \frac{N}{d_l}-1 \\
                                                                                                                                            \end{array}
                                                                                                                                          \right)/(N/d_l)
$$ equivalent classes (each class has exacly $N/d_l$ suitable sequences), for each $j=1,\dots,k_l$.  However, for $i<l$, $St(H_{ij},m,N)$ may contain some $(S)\in St(H_{sj},m,N)$ (and hence all $St(H_{sj},m,N)$), for some $s>i$ and $1\leq j \leq k_s$ which these equivalent classes, $[(S)]$, have order less than $N/d_i$.   Now, for each $1\leq i \leq l$ and $1\leq j \leq k_i$, we define $$\overline{S}(H_{ij}):=\{ (S)\in St(H_{ij},m,N)\,|\, |[(S)]|=N/d_i  \},$$
$$t(N,d_i):=\left(
                                                                                                                                            \begin{array}{c}
                                                                                                                                              \frac{|G|}{d_i}-1 \\
                                                                                                                                              \frac{N}{d_i}-1 \\
                                                                                                                                            \end{array}
                                                                                                                                          \right) \left(
                                                                                                                                            \begin{array}{c}
                                                                                                                                              \frac{m}{d_i}-1 \\
                                                                                                                                              \frac{N}{d_i}-1 \\
                                                                                                                                            \end{array}
                                                                                                                                          \right)
$$
and $C_s(H_{ij})$ to be the set of all subgroups of order $d_s$ of $G$ containing $H_{ij}$, for each $1\leq s\leq l$.
By the above discussion, Proposition \ref{prop41}, \ref{prop43}, Lemma \ref{lem45} and by basic combinatorics, we compute directly that;
\begin{equation}\label{eq 41}
    |\overline{S}(H_{lj})|=t(N,d_l), \quad \hbox{ for each $j=1,\dots,k_l$},
\end{equation}
\begin{equation}\label{eq 42}
    |\overline{S}(H_{(l-1)j}|=t(N,d_{l-1}) -|C_l(H_{(l-1)j})|t(N,d_l), \quad \hbox{ for each $j=1,\dots,k_{l-1}$},
\end{equation}
and, for $2\leq v<l$, $1\leq j \leq k_{l-v}$,
\begin{equation}\label{eq 43}
    |\overline{S}(H_{(l-v)j}|=t(N,d_{l-v})-\left[\sum^{l-1}_{s=l-v+1}\sum_{H\in C_s(H_{(l-v)j})}|\overline{S}(H)| \right]-|C_l(H_{(l-v)j})| t(N,d_l)  .
\end{equation}
Since $\overline{S}(H_{ij})$'s are all distinct sets, there are exactly  $\sum_{i=1}^{l}\frac{d_i}{N}\sum_{j=1}^{k_i}|\overline{S}(H_{ij})|$ equivalent classes of order less than $N$.  So, there must be $\frac{1}{N}\left[ t_N-\sum_{i=1}^{l}\sum_{j=1}^{k_i}|\overline{S}(H_{ij})|\right]$ equivalent classes of order $N$.  Therefore, the following is immediate.

\begin{thm}\label{thm 46} The total number of near-vector spaces $\mathbb{F}^{\oplus m}$ over a finite filed $\mathbb{F}$ is exactly $\sum^{\min(m,|G|)}_{N=1}T(N)$ which each $T(N)$ is explicitly as
$$\frac{t_N}{N}+\frac{1}{N}\sum_{i=1} ^{l} (d_i-1)\sum_{j=1}^{k_i}|\overline{S}(H_{ij})|,$$
where $1<d_1<\cdots<d_l$ are all divisors of $\operatorname{gcd}(m,N,|G|)$ and $H_{ij}\in S_{d_i}(G,N)$, for each $i=1,\dots,l$ and $j=1,\dots,k_i:=|S_{d_i}(G,N)|$.  Here, $|\overline{S}(H_{ij})|$'s can be read from (\ref{eq 41}), (\ref{eq 42}) and (\ref{eq 43}), recursively.
\end{thm}
We see from this theorem that the number of near-vector spaces depends on the subgroups lattice of $G$.  For example, in the case $n=3, p=3$, we have $G_1=U(26)/\left< 3 \right>=\{ 1,5,7,17\}$ and, in the case $n=2, p=5$, we have  $G_2=U(24)/\left< 5 \right>=\{ 1,7,13,19\}$.   Their group structures are illustrated as in the group tables below:
\begin{center}
\begin{tabular}{c|c|c|c|c}
 \hline
  \hline
  $G_1$ & $1$ & $5$ &$ 7$ & $17$ \\
  \hline
  $1$ & $1$ & $5$ & $7$ &$ 17$ \\
  $5$ & $5$ & $17$ & $1$ & $7$ \\
  $7$ & $7$ & $1$ & $17$ & $5$ \\
  $17$ & $17$ & $7$ & $5$ & $1$ \\
  \hline
  \hline
\end{tabular}  \quad \quad and  \quad \quad \begin{tabular}{c|c|c|c|c}
 \hline
  \hline
  $G_2$ & $1$ & $7$ &$ 13$ & $19$ \\
  \hline
  $1$ & $1$ & $7$ & $13$ &$ 19$ \\
  $7$ & $7$ & $1$ & $19$ & $13$ \\
  $13$ & $13$ & $19$ & $1$ & $7$ \\
  $19$ & $19$ & $13$ & $7$ & $1$ \\
  \hline
  \hline
\end{tabular}  
\end{center}
We see that $G_1$ has only one subgroup of order $2$, which is $\{1,17\}$, whereas $G_2$ has three subgroups of order $2$, which are $\{1,7\}$, $\{1,13\}$ and $\{1,19\}$.  By direct calculation (listing all possible suitable sequences and then grouping them), we have tables of the number of near-vector spaces on $\mathbb{F}^{\oplus m}$ corresponding to $G_1$ and $G_2$, with $m=4,\dots,8$ as below:

\begin{center}
\begin{tabular}{c|c|c|c|c|c}
  \hline
  \hline
  $T(N)$; $G_1$ & $m=4$ & $m=5$ & $m=6$ & $m=7$ & $m=8$ \\
  \hline
  $N=1$ & 1 & 1 & 1 & 1 & 1 \\
   $N=2$& 5 & 6 & 8 & 9 & 11 \\
  $N=3$ & 3 & 6 & 10 & 15 & 21 \\
  $N=4$ & 1 & 1 & 3 & 5 & 10 \\
  \hline
  Total  & 10 & 14 & 22 & 30 & 43 \\
  \hline
  \hline
\end{tabular}
\end{center}
and, 
\begin{center}
\begin{tabular}{c|c|c|c|c|c}
  \hline
  \hline
  $T(N)$; $G_2$ & $m=4$ & $m=5$ & $m=6$ & $m=7$ & $m=8$ \\
  \hline
  $N=1$ & 1 & 1 & 1 & 1 & 1 \\
   $N=2$& 6 & 6 & 9 & 9 & 12 \\
  $N=3$ & 3 & 6 & 10 & 15 & 21 \\
  $N=4$ & 1 & 1 & 4 & 5 & 11 \\
  \hline
  Total  & 11 & 14 & 24 & 30 & 45 \\
  \hline
  \hline
\end{tabular}
\end{center}
In fact, this calculation agrees with the calculation using Theorem \ref{thm 46}.

\section*{Acknowledgements}
Some parts of this work were done while the first author was a visiting assistant professor at the Auburn university: he would like to thank Prof. Tin -Yau Tam, for the hospitality.  The first author also expresses his thanks to Naresuan University for financial support on this project (R2560C071).

\bigskip

\address \textbf{Kijti Rodtes} \\

{ Department of Mathematics, Faculty of Science, \\ Naresuan University, and Research Center for Academic Excellent in Mathematics \\ Phitsanulok 65000, Thailand}\\
\email{kijtir@nu.ac.th, \quad \quad and}\\

{ Department of Mathematics and Statistics, \\Auburn University, Alabama 36849, USA}\\
\email{kzr0033@auburn.edu}\\

\address \textbf{Wilasinee Chomjun} \\

{ Department of Mathematics, Faculty of Science,\\ Naresuan University, Phitsanulok 65000, Thailand}\\
\email{wilasinee\_chomjun@hotmail.com }\\


\begin{thebibliography}{20}
\bibitem{JA}
J. Andre, \emph{Lineare Algebra  uber Fastkorpern}, Math. Z. 136 (1974) 295-313.
\bibitem{KTHM}
K.-T. Howell and J.H. Meyer, \emph{Finite dimensional near-vector spaces over finite fields}, Comm. Algebra 38 (2010) 86-93.
\bibitem{KHM}
K.-T Howell and J.H. Meyer,\emph{ Near-vector spaces determined by finite fields}, J. Algebra 398 (2014), 55-62.
\bibitem{AVW}
A.P.J. van der Walt, \emph{Matrix near-rings contained in 2-primitive near-rings with minimal subgroups}, J. Algebra 148 (1992) 296-304.



\end{thebibliography}
\end{document}